\documentclass[a4paper,11pt]{article}
\usepackage{amsmath,amssymb,amsthm,graphics,latexsym,amsfonts}
\usepackage{fancyhdr}
\usepackage{CJK}
\usepackage{graphicx}
\usepackage{indentfirst}
\usepackage{color}
\usepackage{ifpdf}
\usepackage{tikz}
\usepackage{cite}
\usepackage{mathrsfs}
\usepackage{float}
\usepackage{xcolor}
\usepackage{pifont}
\usepackage{amsthm}
\usepackage{algorithm,algcompatible,amsmath}
\usepackage{algpseudocode}
\usepackage[numbers,sort&compress]{natbib}
\usepackage[colorlinks,linkcolor=blue,anchorcolor=blue,citecolor=blue]{hyperref}
\usepackage{booktabs}
\usepackage{cases}
\usepackage{diagbox}
\usepackage{csquotes}

\usetikzlibrary{positioning}

\title{The list $r$-hued coloring of trees and unicyclic graphs
\thanks{Research supported by the Natural Science Foundation of Xinjiang Uygur Autonomous Region (No. 2025D01C35).}}
\author{ Yu Miao$^{a}$, Fengxia Liu$^{a}$\thanks{Corresponding author. Email address: xjulfx@163.com}
 \\
\small $^{a}$College of Mathematics and Systems Science, Xinjiang University, \\
\small Urumqi, Xinjiang, 830017, P.\ R.\ China\\
}
\date{}
\begin{document}
\maketitle

\newtheorem{theorem}{Theorem}[section]
\newtheorem{definition}[theorem]{Definition}
\newtheorem{lemma}[theorem]{Lemma}
\newtheorem{Claim}[theorem]{Claim}
\newtheorem{corollary}[theorem]{Corollary}
\newtheorem{Conjecture}[theorem]{Conjecture}
\newtheorem{remark}[theorem]{Remark}
\newtheorem{proposition}[theorem]{Proposition}
\theoremstyle{plain} \CJKtilde
\newcommand{\D}{\displaystyle}
\newcommand{\DF}[2]{\D\frac{#1}{#2}}
\newtheorem{claim}{Claim}

\begin{abstract}
Let $r$ be a positive integer and $G$ be a graph. The list $r$-hued chromatic number of $G$, denoted by $\chi_{L,r}(G)$, is the smallest integer $k$, such that for each $k$-list $L$ of $G$, $G$ has an $(L,r)$-coloring. It is proved in [Discrete Math. 306 (16) (2006) 1997-2004] that every tree $G$ satisfies $\chi_{r}(G)=\min\{r,\Delta(G)\}+1$. It is known that every cycle graph $C_{n}$ with order $n$ has $\chi_{L,r}(C_{n})=\chi_{r}(C_{n})$.
The main results are the following:

 $(1)$ If $G$ is a tree, then $\chi_{L,r}(G)=\min\{r,\Delta(G)\}+1$;

 $(2)$ Let $G$ be a unicyclic graph which is not isomorphic to the cycle $C_{n}$. If $n\neq 5$ and $r\geq3$, then $\chi_{L,r}(G)=\min\{r,\Delta(G)\}+1$; otherwise, $\min\{r,\Delta(G)\}+1\leq\chi_{L,r}(G)\leq\min\{r,\Delta(G)\}+2$.

\end{abstract}



\section{Introduction}
Let $\mathbb{N}$ denote the set of all natural numbers and $2^{\mathbb{N}}$ be the power set of $\mathbb{N}$. All graphs considered in this paper are simple and finite, and undefined terminologies and notation will follow \cite{BoMu01}. As in \cite{BoMu01}, we use $N_{G}(v)$, $d_{G}(v)$, $\delta(G)$ and $\Delta{(G)}$ to denote the set of neighbors of a vertex $v$, the degree of a vertex $v$, the minimum degree of a graph $G$ and the maximum degree of a graph $G$.
A stable set is a set of vertices no two of which are adjacent. A stable set in a graph is a maximum stable set when there is no larger stable set in the graph. The number of vertices in a maximum stable set of a graph $G$ is referred to as the stability number of $G$, which is denoted by $\alpha(G)$.
Denote the distance of two vertices $u$ and $v$ in $G$ as $d_{G}(u, v)$. A list of a graph $G$ is an assignment $L: V(G)\rightarrow 2^{\mathbb{N}}$ that assigns every $v\in V(G)$ a list $L(v)$ of colors available at $v$. Let $k$ and $r$ be positive integers. A list $L$ of a graph $G$ is a {\em $k$-list} if $|L(v)|= k$ for any $v\in V(G)$. For a mapping $c: V(G)\rightarrow \mathbb{N}$, we define the following conditions on $c$.

\par $(C1)$ The proper coloring condition: $c(u)\neq c(v)$ for every edge $uv\in E(G)$;
\par $(C2)$ The $r$-hued coloring condition: $|c(N_{G}(v))| \geq \min\{d_{G}(v), r\}$ for any $v\in V(G)$;
\par $(C3)$ The list coloring condition: $c(v)\in L(v)$, for every $v\in V(G)$.

\par A mapping $c: V(G)\rightarrow\{1, 2, \ldots, k\}$ satisfying condition $(C1)$ is a proper {\em $k$-coloring} of $G$; one that satisfies both $(C1)$ and $(C2)$ is a {\em $(k,r)$-coloring} ({\em $r$-hued $k$-coloring}) of $G$. If $c: V(G)\rightarrow \mathbb{N}$ and satisfies both $(C1)$ and $(C3)$, then $c$ is an {\em $L$-coloring} of $G$. Furthermore, if an $L$-coloring $c$ also satisfies $(C2)$, then $c$ is an {\em $(L,r)$-coloring} ({\em list $r$-hued $L$-coloring}) of $G$.

\par The {\em chromatic number} of $G$, denoted by $\chi(G)$, is the smallest integer $k$ such that $G$ has a proper $k$-coloring. The {\em $r$-hued chromatic number} of $G$, denoted by {\em $\chi_{r}(G)$}, is the smallest $k$ such that $G$ has a $(k,r)$-coloring. The {\em list chromatic number} of $G$, denoted by {\em $\chi_{L}(G)$}, is the smallest integer $k$ such that for any $k$-list $L$ of $G$, $G$ has an $L$-coloring. The {\em list $r$-hued chromatic number} of $G$, denoted by {\em $\chi_{L,r}(G)$}, is the smallest integer $k$, such that for any $k$-list $L$ of $G$, $G$ has an $(L,r)$-coloring. By definition, $\chi_{L}(G)=\chi_{L,1}(G)$. The {\em square of a graph} $G$, denoted by $G^{2}$, has $V(G^{2})=V(G)$, where $uv\in E(G^{2})$ if and only if the distance between $u$ and $v$
in $G$ is at most $2$. By definition, for a graph $G$ with $\Delta=\Delta(G)$, $\chi(G^{2})=\chi_{\Delta}(G)$.

\par The concept of $r$-hued coloring was first introduced in the dissertation of Bruce Montgomery in \cite{SHI29} (see also  \cite{{LMP30}}). The concept of list $2$-hued coloring in graphs was first introduced in \cite{SMS13} under the term \enquote{list dynamic coloring}. Later, in \cite{CFLSS15}, the general notion of list $r$-hued coloring for arbitrary values of $r$ was formally defined.

\par It follows from the definition that for every graph $G$, if $1\leq i \leq j$, then

\begin{equation}\label{eq:1'}
\chi_{i}(G)\leq \chi_{j}(G);~\chi_{L,i}(G)\leq \chi_{L,j}(G).
\end{equation}

It is observed in \cite{LLMS05} that for any graph $G$ and for any positive integer $r$, the following always holds:

\begin{equation}\label{eq:2'}
\min\{r,\Delta(G)\}+1\leq \chi_{r}(G)\leq \chi_{L,r}(G)\leq |V(G)|.
\end{equation}

\par For an integer $n\geq3$, let $C_{n}$ denote a cycle of order $n$. If $r=1$, then it is known that

\begin{equation}\label{eq:3'}
\chi(C_{n})=
\begin{cases}
2 & n\equiv0~(mod~2); \\
3 & n\equiv1~(mod~2). \\
\end{cases}
\end{equation}

\par The $r$-hued chromatic number of trees was determined in \cite{LLMS05} in 2006.

\begin{theorem} \rm\text{(}\cite{LLMS05}\text{)}\label{Theorem1.1'}.
\itshape{If $G$ is a tree with $|V(G)|\geq3$, then $\chi_{r}(G)=\min\{r,\Delta(G)\}+1$}.
\end{theorem}

\par In 2024,  the list chromatic number of trees and cycles were determined in \cite{MJVU21}. If a tree is nontrivial, it must contain a vertex of degree exactly one. Such a vertex is called a leaf of the tree.

\begin{theorem} \rm\text{(}\cite{MJVU21}\text{)}\label{Theorem2.1'}.
\itshape{For every nontrivial tree $T$, $\chi_{L}(T)=2$}.
\end{theorem}

\begin{theorem} \rm\text{(}\cite{MJVU21}\text{)}\label{Theorem2.2'}.
\itshape{For every integer $n\geq3$, let $C_{n}$ denote the cycle of order $n$. Then}
\end{theorem}

\begin{equation}
\chi_{L}(C_{n})=
\begin{cases}
2 & n\equiv0~(mod~2); \\
3 & n\equiv1~(mod~2). \\
\end{cases}
\end{equation}

\par In \cite{JLJZ07}, Jia et al. determined the value of $\chi_{L,r}(G)$ for star graphs, which are trees of diameter $2$.
\begin{theorem} \rm\text{(}\cite{JLJZ07}\text{)}\label{Theorem1.4'}.
If $G$ is a $K_{1,n-1}$, then $\chi_{L,r}(G)=\min\{r,\Delta(G)\}+1$.
\end{theorem}

\par As star graphs are special instances of trees, a natural question arises: does an analogue of Theorem \ref{Theorem1.1'} hold for $\chi_{L,r}(G)$ of trees? Theorem \ref{TH1.5'} gives an affirmative answer to this question, which extends Theorem \ref{Theorem1.1'}.

\begin{theorem} \label{TH1.5'}
If $G$ is a tree, then $\chi_{L,r}(G)=\min\{r,\Delta(G)\}+1$.
\end{theorem}

\par In 2009, Akbari, Ghanbari, and Jahanbekam \cite{SMS13} observed $\chi_{L,2}(C_{n})=\chi_{2}(C_{n})$. By \text{(}\ref{eq:1'}\text{)} and \text{(}\ref{eq:2'}\text{)}, $\chi_{2}(C_{n})\leq \chi_{L,2}(C_{n})=\chi_{L,3}(C_{n})=\cdots $. Therefore, the following theorem is obtained.

\begin{theorem} \rm\text{(}{\cite{SMS13,LLMS05}}\text{)}\label{theorem1.2'}.
\itshape{If $r\geq2$, then}
$$
\chi_{L,r}(C_{n})=\chi_{r}(C_{n})=
\begin{cases}
\text 5 \quad if\    n=5; \\
\text 3 \quad if\  n\equiv0~(mod~3); \\
\text 4 \quad otherwise.
\end{cases}
$$
\end{theorem}

\par A connected graph $G$ is called a {\em $c$-cyclic graph} if $c=|E(G)|-|V(G)|+1$. In particular, $G$ is called a tree or a {\em unicyclic graph} if $c=0$ or $1$, respectively. Thus every unicyclic graph $G$ with $\delta(G)=1$ consists of a unique cycle and a nontrivial forest.

\par As Theorem \ref{TH1.5'} and \ref{theorem1.2'} already show the values of the list $r$-hued chromatic number of trees and cycles, it is natural to consider what the value of the list $r$-hued chromatic number of a unicyclic graph is. For this problem, we have arrived at the following result.

\begin{theorem} \label{TH1.6'}
Let $r$ be an integer, $G$ be a unicyclic graph with $\delta(G)=1$, $k=\min\{r,\Delta(G)\}+1$, and the cycle of $G$ have order $n\geq3$. Then

$$
\chi_{L,r}(G)\in
\begin{cases}
\{k\} & n\neq5~ and~ r\geqslant3; \\
\{k,k+1\} & n=5~ or~ r=2. \\
\end{cases}
$$
\end{theorem}

\par For further insights and additional findings on this topic, as well as for other pertinent conclusions and references, the reader is advised to consult the literature cited as \cite{Mont14, Mont15, LLMS08, CFLX09}.
\par The proof of Theorem \ref{TH1.5'} will be presented in Section $2$, while the justification for Theorem \ref{TH1.6'} will be provided in Section 3.

\section{The list $r$-hued chromatic number of a tree}
Theorem \ref{TH1.5'} will be proved in this section. We start with a lemma.

\begin{lemma}\label{Lemma2.2'}
Let $G$ be a tree, $r$ be an integer with $r\geq2$, $k=\min\{r, \Delta(G)\}+1$, and let $L$ be a $k$-list of $G$. For any vertex $v_{0}$ of $G$, and for any $a_{0}\in L(v_{0})$, there exists an $(L,r)$-coloring $c$ of $G$ such that $c(v_{0})=a_{0}$.
\end{lemma}
\noindent\textbf{Proof}. We prove by induction on $n=|V(G)|$. By definition, Lemma \ref{Lemma2.2'} holds for $n\in\{1,2\}$ trivially. Assume that $n\geq3$ and Lemma \ref{Lemma2.2'} holds for smaller values of $n$.
\par For any $k$-list $L$ of $G$, we arbitrarily pick a vertex $v_{0}$ of $G$ and a color $a_{0}\in L(v_{0})$. We are to prove that

\begin{equation}\label{eq:7'}
    \text{there exists an $(L,r)$-coloring $c$ of $G$ such that $c(v_{0})=a_{0}$.} \tag{5}
\end{equation}
\setcounter{equation}{5}

\par Since $G$ is a tree and $n\geq3$, there must be at least two leaves. Choose a leaf distinct from $v_{0}$, denoted by $u_{0}$, and its unique neighbor in $G$ is denoted by $u_{1}$. We assume $N_{G}(u_{1})=\{u_{0}, x_{1}, x_{2},\ldots, x_{t}\}$ where $t=d_{G}(u_{1})-1$ and $H=G-u_{0}$. For any $k$-list $L$ of graph $G$, let $L_{1}$ be the list of $L$ restricted to $V(H)$, and so $L_{1}$ is a $k$-list of $H$. Then for any $v\in V(H)$, $|L_{1}(v)|=k=\min\{r, \Delta(G)\}+1\geq\min\{r, \Delta(H)\}+1$. Since $v_{0}\neq u_{0}$, we must have $v_{0}\in V(H)$.
Since $|V(H)|<|V(G)|$, by induction, for the vertex $v_{0}$ of $H$, and for $a_{0}\in L(v)$, there exists an $(L_{1},r)$-coloring $c_{1}$ of $H$ such that $c_{1}(v_{0})=a_{0}$.

\noindent\textbf{Case 1:} $d_{H}(u_{1})< r$. Define:\\
\begin{equation}\label{eq:4'}
c(v)\in
\begin{cases}
L(u_{0})-\{c(u_{1}), c(x_{1}), c(x_{2}),\ldots, c(x_{t})\} & v=u_{0}; \\
\{c_{1}(v)\} & v\neq u_{0}. \\
\end{cases}
\end{equation}

\par Since $|\{c(u_{1}), c(x_{1}), c(x_{2}),\ldots, c(x_{t})\}|=t+1=d_{G}(u_{1})\leq\Delta(G)$ and $d_{G}(u_{1})=d_{H}(u_{1})+1\leq r$, it follows that $|\{c(u_{1}), c(x_{1}), c(x_{2}),\ldots, c(x_{t})\}|\leq \min\{r,\Delta(G)\}=k-1$. Therefore, we have $|L(u_{0})-\{c(u_{1}), c(x_{1}), c(x_{2}),\ldots, c(x_{t})\}|\geq 1$. Hence it is possible to define $c(u_{0})\in L(u_{0})-\{c(u_{1}), c(x_{1}), c(x_{2}),\ldots, c(x_{t})\}$.
For any $v\in V(G)$, by \text{(}\ref{eq:4'}\text{)}, we obtain a coloring $c$ that satisfies conditions $(C1)$ and $(C3)$. Furthermore, condition $(C2)$ is also satisfied since the colors on the vertices in the neighborhood of $u_{1}$ are different, and all other vertices have received at least $\min\{d_{G}(v), r\}$ colors in their neighborhood. Thus, \text{(}$\ref{eq:7'}$\text{)} holds.

\noindent\textbf{Case 2:} $d_{H}(u_{1})\geq r$. Define:\\
\begin{equation}\label{eq:5'}
c(v)\in
\begin{cases}
L(u_{0})-\{c(u_{1})\} & v=u_{0}; \\
\{c_{1}(v)\} & v\neq u_{0}. \\
\end{cases}
\end{equation}

\par Since $|L(u_{0})-\{c(u_{1})\}|\geq 1$, it is possible to define $c(u_{0})\in L(u_{0})-\{c(u_{1})\}$. For any $v\in V(G)$, by \text{(}\ref{eq:5'}\text{)}, we obtain a coloring $c$ that satisfies conditions $(C1)$ and $(C3)$. Furthermore, condition $(C2)$ is also satisfied since each vertex $v$ receives at least $\min\{d_{G}(v), r\}$ colors in its neighborhood. Thus, \text{(}$\ref{eq:7'}$\text{)} holds.

\par Thus, for any vertex $v_{0}$ of $G$, and for any $a_{0}\in L(v_{0})$, there exists an $(L,r)$-coloring $c$ of $G$ such that $c(v_{0})=a_{0}$.
$\hfill\square$

\begin{corollary}\label{corollary 2.2}
If $G$ is a tree, and $r\geq2$, then $\chi_{L,r}(G)=\min\{r,\Delta(G)\}+1$.
\end{corollary}

\noindent\textbf{Proof}. By \text{(}\ref{eq:2'}\text{)}, it follows that $\chi_{L,r}(G)\geq \min\{r,\Delta(G)\}+1$. Next, we only need to prove that $\chi_{L,r}(G)\leq \min\{r,\Delta(G)\}+1$. For any list $L$ with $|L|=\min\{r,\Delta(G)\}+1$, by Lemma \ref{Lemma2.2'}, we obtain a coloring $c$ that is indeed an $(L, r)$-coloring of the graph $G$. Hence $\chi_{L,r}(G)\leq \min\{r,\Delta(G)\}+1$. Thus, for $r\geq2$, we have $\chi_{L,r}(G)=\min\{r,\Delta(G)\}+1$.
$\hfill\square$

\par If $G$ is a trivial tree, then $\chi_{L}(G)=\chi_{L,1}(G)=1$. If $G$ is a nontrivial tree, then, by Theorem \ref{Theorem2.1'}, $\chi_{L}(G)=\chi_{L,1}(G)=2$. Therefore, if $G$ is a tree, then $\chi_{L}(G)=\chi_{L,1}(G)=\min\{1, \Delta(G)\}+1$. Thus, together with Corollary \ref{corollary 2.2}, we conclude that Theorem \ref{TH1.5'} is justified.

\section{The list $r$-hued chromatic number of a unicyclic graph}

\par Firstly, we give the following definitions. We say that a vertex $v$ of $G$ is a leaf neighbor of $u$ if $v\in N_{G}(u)$ and $d_{G}(v)=1$. We define $\mathcal{T}=\{T_{u}~|$ $T_{u}$ is a tree such that $T_{u}$ has a root vertex $u$ with degree at least three in $G$, and $u$ has at least two leaf neighbors$\}$.

\par For an integer $n\geq3$, define $\mathcal{G}_{n}=\{G~|$ $G$ is a unicyclic graph with $\delta(G)=1$, and with $C_{n}$ being the unique cycle of $G$ $\}$. For each graph $G \in \mathcal{G}_{n}$ with $C_{n}$ being the cycle in $G$, denote $V(C_{n})=\{v_{1},v_{2},\ldots,v_{n}\}$ and write $C_{n}=v_{1}v_{2}\cdots v_{n}v_{1}$. Define $M(G)=\{v~|~v\in V(C_{n})$ and $d_{G}(v)\geq3\}$. By definition, if $G \in \mathcal{G}_{n}$, then $M(G)\neq \emptyset$ and $\Delta(G)\geq3$. If $v_{i}\in M(G)$, then there exists a maximal tree $S_{i}$ of $G\setminus E(C_{n})$ with a root vertex $v_{i}$. Let $T_{v_{i}}=G[V(S_{i})\cup \{v_{i-1}, v_{i+1}\}]$ be the induced subgraph of $G$, $\{v_{i-1}, v_{i}, v_{i+1}\}\subseteq V(C_{n})$, and vertices $v_{i-1}$ and $v_{i+1}$ are leaf neighbors of $v_{i}$ in $T_{v_{i}}$. If $d_{G}(v_{i-1})=d_{G}(v_{i+1})=2$ for every $v_{i}\in M(G)$, then $M(G)$ is a stable set and $\alpha(C_{n})=\lfloor\frac{n}{2}\rfloor$. By definition, when $G\in \mathcal{G}_{n}$, $v_{i}\in M(G)$, $T_{v_{i}}\in \mathcal{T}$ (see Figure \ref{fig1}).

\par In this section, let $r$ be a positive integer and $G\in \mathcal{G}_{n}$. To prove Theorem \ref{TH1.6'}, we will discuss two cases when $r=1$ and when $r\geq 2$. In the case $r\geq 2$, we further distinguish between the cases when $n=5$ and when $n\neq5$ for detailed analysis.

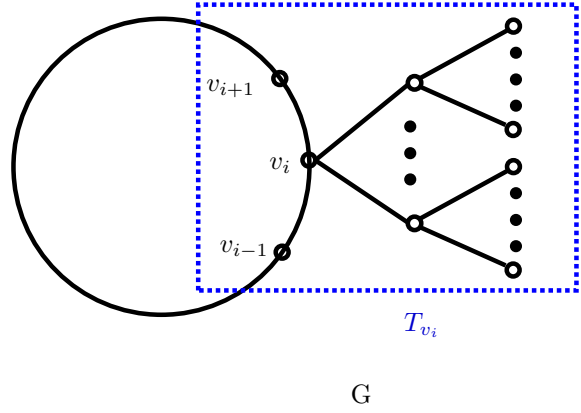
\begin{figure}
\centering
\begin{tikzpicture}[x=1.00mm, y=1.00mm, inner xsep=0pt, inner ysep=0pt, outer xsep=0pt, outer ysep=0pt]
\path[line width=0mm] (123.04,-75.35) rectangle +(78.45,57.45);
\definecolor{L}{rgb}{0,0,0}
\path[line width=0.60mm, draw=L] (144.60,-41.34) circle (19.57mm);
\path[line width=0.60mm, draw=L] (160.56,-52.63) circle (0.85mm);
\path[line width=0.60mm, draw=L] (164.10,-40.46) circle (0.85mm);
\path[line width=0.60mm, draw=L] (160.25,-29.67) circle (0.85mm);
\path[line width=0.60mm, draw=L] (191.11,-36.41) circle (0.85mm);
\path[line width=0.60mm, draw=L] (178.07,-30.24) circle (0.89mm);
\path[line width=0.60mm, draw=L] (191.15,-22.74) circle (0.85mm);
\path[line width=0.60mm, draw=L] (165.01,-40.42) -- (177.22,-30.46);
\path[line width=0.60mm, draw=L] (165.14,-40.61) -- (177.14,-48.68);
\path[line width=0.60mm, draw=L] (178.66,-29.70) -- (190.41,-23.26);
\path[line width=0.60mm, draw=L] (178.87,-30.78) -- (190.11,-35.80);
\definecolor{F}{rgb}{0,0,0}
\path[line width=0.60mm, draw=L, fill=F] (177.50,-39.54) circle (0.50mm);
\path[line width=0.60mm, draw=L, fill=F] (177.50,-36.01) circle (0.50mm);
\path[line width=0.60mm, draw=L, fill=F] (177.50,-43.01) circle (0.50mm);
\path[line width=0.60mm, draw=L, fill=F] (191.48,-29.79) circle (0.50mm);
\path[line width=0.60mm, draw=L, fill=F] (191.48,-26.26) circle (0.50mm);
\path[line width=0.60mm, draw=L, fill=F] (191.48,-33.26) circle (0.50mm);
\draw(158.89,-41.69) node[anchor=base west]{\fontsize{9.92}{11.91}\selectfont $v_{i}$};
\draw(152.33,-52.60) node[anchor=base west]{\fontsize{9.92}{11.91}\selectfont $v_{i-1}$};
\draw(150.57,-31.34) node[anchor=base west]{\fontsize{9.92}{11.91}\selectfont $v_{i+1}$};
\definecolor{L}{rgb}{0,0,1}
\path[line width=0.60mm, draw=L, dash pattern=on 0.60mm off 0.50mm] (149.46,-57.69) [rotate around={360:(149.46,-57.69)}] rectangle +(50.02,37.79);
\definecolor{T}{rgb}{0,0,0.804}
\draw[T] (176.70,-63.01) node[anchor=base west]{\fontsize{9.92}{11.91}\selectfont \textcolor[rgb]{0, 0, 0.80392}{$T_{v_{i}}$}};
\definecolor{T}{rgb}{0,0,0}
\draw[T] (169.64,-72.60) node[anchor=base west]{\fontsize{9.92}{11.91}\selectfont \textcolor[rgb]{0, 0, 0}{G}};
\definecolor{L}{rgb}{0,0,0}
\path[line width=0.60mm, draw=L] (191.11,-54.99) circle (0.85mm);
\path[line width=0.60mm, draw=L] (178.07,-48.82) circle (0.89mm);
\path[line width=0.60mm, draw=L] (191.15,-41.32) circle (0.85mm);
\path[line width=0.60mm, draw=L] (178.66,-48.27) -- (190.41,-41.84);
\path[line width=0.60mm, draw=L] (178.87,-49.36) -- (190.11,-54.38);
\path[line width=0.60mm, draw=L, fill=F] (191.48,-48.37) circle (0.50mm);
\path[line width=0.60mm, draw=L, fill=F] (191.48,-44.83) circle (0.50mm);
\path[line width=0.60mm, draw=L, fill=F] (191.48,-51.84) circle (0.50mm);
\end{tikzpicture}%
\caption{The graph $T_{v_{i}}$}
\label{fig1}
\end{figure}

\par Let $r$ be a positive integer. For any graph $G$, let $k=k(r, G)=\min\{r,\Delta(G)\}+1$, and this definition will be consistently used henceforth. By Theorem \ref{TH1.5'}, for any $T_{v_{i}}\in \mathcal{T}$, $\chi_{L,r}(T_{v_{i}})=\min\{r,\Delta(T_{v_{i}})\}+1$.

\begin{lemma}\label{Lemma3.1'}
Let $T_{v_{i}}\in \mathcal{T}$, $k=k(r, T_{v_{i}})$, $v_{i-1}$ and $v_{i+1}$ be two leaf neighbors of $v_{i}$. For any $k$-list $L$ of $T_{v_{i}}$, each of the following holds.\\
$(i)$ If there exist $a_{1}$, $a_{2}$ and $a_{3}$ such that $a_{1}\in L(v_{i})$, $a_{2}\in L(v_{i-1})$ and $a_{3}\in L(v_{i+1})$, where $a_{1}, a_{2}$ and $a_{3}$ are mutually distinct, then $T_{v_{i}}$ has an $(L,r)$-coloring $c$ satisfying $c(v_{i})=a_{1}$, $c(v_{i-1})=a_{2}$ and $c(v_{i+1})=a_{3}$;\\
$(ii)$ If $d_{T_{v_{i}}}(v_{i})>r$, and there exist $a_{1}$ and $a_{2}$ such that $a_{1}\in L(v_{i})$, $a_{2}\in L(v_{i-1})\cap L(v_{i+1})$ and $a_{1}\neq a_{2}$, then $T_{v_{i}}$ has an $(L,r)$-coloring $c$ satisfying $c(v_{i})=a_{1}$ and $c(v_{i-1})=c(v_{i+1})=a_{2}$.
\end{lemma}

\noindent\textbf{Proof}. We prove the lemma by induction on $|V(T_{v_{i}})|$. Since $T_{v_{i}}\in \mathcal{T}$, by the definition of $\mathcal{T}$, $|V(T_{v_{i}})|\geq4$. If $|V(T_{v_{i}})|=4$, then $T_{v_{i}}=K_{1,3}$ and $k=\min\{r, 3\}+1$. Define $N_{K_{1,3}}(v_{i})=\{v_{0}, v_{i-1}, v_{i+1}\}$.

\par Proof of $(i)$. Suppose that $L$ is an arbitrary $k$-list of $K_{1,3}$. We consider the following cases.

\par If $d_{T_{v_{i}}}(v_{i})=3>r$, then $k=\min\{r, 3\}+1=r+1$, and we define:

\begin{equation}\label{eq:3.1.1'}
c(v)\in
\begin{cases}
\{a_{1}\} & v=v_{i}; \\
\{a_{2}\} & v=v_{i-1}; \\
\{a_{3}\} & v=v_{i+1}; \\
L(v_{0})-\{a_{1}\} & v=v_{0}. \\
\end{cases}
\end{equation}

\par If $d_{T_{v_{i}}}(v_{i})=3\leq r$, then $k=\min\{r, 3\}+1=4$, and we define:

\begin{equation}\label{eq:3.1.2'}
c(v)\in
\begin{cases}
\{a_{1}\} & v=v_{i}; \\
\{a_{2}\} & v=v_{i-1}; \\
\{a_{3}\} & v=v_{i+1}; \\
L(v_{0})-\{a_{1}, a_{2}, a_{3}\} & v=v_{0}. \\
\end{cases}
\end{equation}

\par By \text{(}\ref{eq:3.1.1'}\text{)} and \text{(}\ref{eq:3.1.2'}\text{)}, $(i)$ of Lemma \ref{Lemma3.1'} holds for $|V(T_{v_{i}})|=4$.

\par Assume that $|V(T_{v_{i}})|\geq5$ and $(i)$ of Lemma \ref{Lemma3.1'} holds for smaller values of $|V(T_{v_{i}})|$. For any $k$-list $L$ of $T_{v_{i}}$, we need to prove that there exists an $(L,r)$-coloring $c$ of $T_{v_{i}}$ such that $c(v_{i})=a_{1}$, $c(v_{i-1})=a_{2}$ and $c(v_{i+1})=a_{3}$.

\par Since $T_{v_{i}}\in \mathcal{T}$ and $|V(T_{v_{i}})|\geq5$, by the definition of $\mathcal{T}$, $T_{v_{i}}$ must have at least three leaves. Thus $T_{v_{i}}$ has a leaf $u_{0}$ distinct from $v_{i-1}$ and $v_{i+1}$. Let $u_{1}$ be the unique neighbor of $u_{0}$ in $T_{v_{i}}$. Denote $N_{T_{v_{i}}}(u_{1})=\{u_{0}, x_{1}, x_{2},\ldots, x_{t}\}$ with $t=d_{T_{v_{i}}}(u_{1})-1$, and let $H=T_{v_{i}}-u_{0}$. For any $k$-list $L$ of graph $T_{v_{i}}$, let $L_{1}$ be the list of $L$ restricted to $H$. Then $L_{1}$ is a $k$-list of $H$. Since $u_{0}\notin \{v_{i-1}, v_{i}, v_{i+1}\}$, we must have $\{v_{i-1}, v_{i}, v_{i+1}\}\subseteq V(H)$.
Since $|V(H)|<|V(T_{v_{i}})|$, by induction, for the vertices $v_{i-1}, v_{i}, v_{i+1}\in V(H)$, and for $a_{1}\in L(v_{i})$, $a_{2}\in L(v_{i-1})$, $a_{3}\in L(v_{i+1})$, there exists an $(L_{1},r)$-coloring $c_{1}$ of $H$ such that $c_{1}(v_{i})=a_{1}$, $c_{1}(v_{i-1})=a_{2}$ and $c_{1}(v_{i+1})=a_{3}$.

\par Similar to the proofs for Case $1$ and Case $2$ of Lemma \ref{Lemma2.2'}, for any $k$-list $L$ of $T_{v_{i}}$, if there exist colors $a_{1}\in L(v_{i})$, $a_{2}\in L(v_{i-1})$, $a_{3}\in L(v_{i+1})$, where $a_{1}, a_{2}$ and $a_{3}$ are mutually distinct, then there exists an $(L,r)$-coloring $c$ of $T_{v_{i}}$ such that $c(v_{i})=a_{1}$, $c(v_{i-1})=a_{2}$ and $c(v_{i+1})=a_{3}$.

\par Proof of $(ii)$. Since $\Delta(T_{v_{i}})\geq d_{T_{v_{i}}}(v_{i})>r\geq1$, $k=\min\{r, \Delta(T_{v_{i}})\}+1=r+1$. Suppose that $|V(T_{v_{i}})|=4$. Then there exists a vertex $v_{0}\in V(T_{v_{i}})-V(C_{n})$ such that $T_{v_{i}}\cong K_{1,3}$ with $v_{i}$ being the degree $3$ vertex. Let $L$ be an arbitrary $k$-list of $K_{1,3}$, we consider the following cases.
\par When $r=1$, $k=r+1=2$, we define:

\begin{equation}\label{eq:3.3.1'}
c(v)\in
\begin{cases}
\{a_{1}\} & v=v_{i}; \\
\{a_{2}\} & v=v_{i-1}; \\
\{a_{2}\} & v=v_{i+1}; \\
L(v_{0})-\{a_{1}\} & v=v_{0}. \\
\end{cases}
\end{equation}

\par When $r=2$, $k=r+1=3$, we define:

\begin{equation}\label{eq:3.3.2'}
c(v)\in
\begin{cases}
\{a_{1}\} & v=v_{i}; \\
\{a_{2}\} & v=v_{i-1}; \\
\{a_{2}\} & v=v_{i+1}; \\
L(v_{0})-\{a_{1}, a_{2}\} & v=v_{0}. \\
\end{cases}
\end{equation}

\par By \text{(}\ref{eq:3.3.1'}\text{)} and \text{(}\ref{eq:3.3.2'}\text{)}, $(ii)$ of Lemma \ref{Lemma3.1'} holds for $|V(T_{v_{i}})|=4$.
\par Assume that $|V(T_{v_{i}})|\geq5$ and $(ii)$ of Lemma \ref{Lemma3.1'} holds for smaller values of $|V(T_{v_{i}})|$. Similar to the proof for $(i)$ of Lemma \ref{Lemma3.1'}, for any $k$-list $L$ of $T_{v_{i}}$, there exists an $(L,r)$-coloring $c$ of $T_{v_{i}}$ such that $c(v_{i})=a_{1}$, $c(v_{i-1})=c(v_{i+1})=a_{2}$.
$\hfill\square$

\begin{theorem}\label{Theorem3.3'}
If $G\in \mathcal{G}_{n}$, then
$$
\chi_{L}(G)=\chi_{L,1}(G)=
\begin{cases}
2 & n\equiv0(mod~2); \\
3 & n\equiv1(mod~2). \\
\end{cases}
$$
\end{theorem}

\noindent\textbf{Proof}. We consider the following cases.

\noindent\textbf{Case 1:} $n\equiv0(mod~2)$.
\par By \text{(}\ref{eq:2'}\text{)}, it follows that $\chi_{L}(G)\geq 2$. Next, we only need to prove that $\chi_{L}(G)\leq 2$.

\par By Theorem \ref{Theorem2.2'}, for every even integer $n\geq4$, $\chi_{L}(C_{n})=2$. For any $2$-list $L$ of $G$, there exists an $L$-coloring $c_{0}$ of $C_{n}$ such that $1\leq |c_{0}(N_{C_{n}}(v_{i}))|\leq2$ for every vertex $v_{i}\in M(G)$. Moreover, we have $d_{G}(v_{i})\geq3$. Since $r=1$, $d_{G}(v_{i})\geq3>1=r$. For any $v_{i}\in M(G)$, there exists an $L$-coloring $c_{i}$ of $T_{v_{i}}$ such that $c_{i}(v_{i})=c_{0}(v_{i})$, $c_{i}(v_{i-1})=c_{0}(v_{i-1})$ and $c_{i}(v_{i+1})=c_{0}(v_{i+1})$. The existence is guaranteed as follows: when $c_{0}(v_{i-1})\neq c_{0}(v_{i+1})$, it follows from $(i)$ of Lemma \ref{Lemma3.1'}; when $c_{0}(v_{i-1})=c_{0}(v_{i+1})$, it follows from $(ii)$ of Lemma \ref{Lemma3.1'}.

Define:

\begin{equation}\label{eq:8'}
c(v)=
\begin{cases}
c_{0}(v) & v\in V(C_{n}); \\
c_{i}(v) & v\in V(T_{v_{i}}). \\
\end{cases}
\end{equation}

\par By the choices of $c_{0}$ and $c_{i}$, for each $i$, a coloring $c$ of $G$ defined by \text{(}\ref{eq:8'}\text{)} satisfies conditions $(C1)$ and $(C3)$. Thus $c$ is indeed an $(L, 1)$-coloring of the graph $G$. Hence $\chi_{L}(G)\leq 2$. Therefore, we have $\chi_{L}(G)=2$.

\noindent\textbf{Case 2:} $n\equiv1(mod~2)$.
\par Since $G\in \mathcal{G}_{n}$ and $C_{n}$ is the unique cycle of $G$, by \text{(}\ref{eq:3'}\text{)}, $\chi(G)\geq \chi(C_{n})=3$. Thus by \text{(}\ref{eq:2'}\text{)}, we have $\chi_{L}(G)\geq\chi(G)\geq3$. Next, we only need to prove that $\chi_{L}(G)\leq 3$.

\par By Theorem \ref{Theorem2.2'}, for every odd integer $n\geq3$, $\chi_{L}(C_{n})=3$. For any $3$-list $L$ of $G$, there exists an $L$-coloring $c_{0}$ of $C_{n}$ such that $1\leq |c_{0}(N_{C_{n}}(v_{i}))|\leq2$ for every vertex $v_{i}\in M(G)$. Similarly, for any $v_{i}\in M(G)$, by Lemma \ref{Lemma3.1'}, there exists an $L$-coloring $c_{i}$ of $T_{v_{i}}$ such that $c_{i}(v_{i})=c_{0}(v_{i})$, $c_{i}(v_{i-1})=c_{0}(v_{i-1})$ and $c_{i}(v_{i+1})=c_{0}(v_{i+1})$. As $c_{0}$ and $c_{i}$ agree on these vertices, by \text{(}\ref{eq:8'}\text{)}, we obtain a coloring $c$ of $G$ that satisfies conditions $(C1)$ and $(C3)$. Thus $c$ is indeed an $L$-coloring of the graph $G$. Hence $\chi_{L}(G)\leq 3$. Therefore, we have $\chi_{L}(G)=3$.
$\hfill\square$

\begin{lemma}\label{3.1-3.2}
Let $G\in \mathcal{G}_{n}$, $r\geq1$ and $l\geq k$. Suppose that for any $l$-list $L$ of $G$, there exists an $(L,1)$-coloring $c_{0}$ of $C_{n}$ such that $|c_{0}(N_{C_{n}}(v_{i}))|=2$ for every vertex $v_{i}\not\in M(G)$. If the coloring $c$ of $G$ is obtained by extending the $(L,1)$-coloring $c_{0}$ using Lemma \ref{Lemma3.1'} and \text{(}\ref{eq:8'}\text{)}, then $c$ is an $(L,r)$-coloring of the graph $G$. Hence $\chi_{L,r}(G)\leq l$.
\end{lemma}

\noindent\textbf{Proof}. For any $l$-list $L$ of $G$, we have $c_{0}$ is an $(L, 1)$-coloring of $C_{n}$. For any $v_{i}\in M(G)$, by Lemma \ref{Lemma3.1'}, there exists an $(L,r)$-coloring $c_{i}$ of $T_{v_{i}}$ such that $c_{i}(v_{i})=c_{0}(v_{i})$, $c_{i}(v_{i-1})=c_{0}(v_{i-1})$ and $c_{i}(v_{i+1})=c_{0}(v_{i+1})$.

\par For any $v\in V(G)$, by \text{(}\ref{eq:8'}\text{)}, we obtain a coloring $c$ of $G$ that satisfies conditions $(C1)$ and $(C3)$. For any $G\in \mathcal{G}_{n}$, define $N(G)=\{u|u\in V(C_{n})~and~u\not\in M(G)\}$. If $u\in N(G)$, then $d_{G}(u)=2$ and $|c_{0}(N_{C_{n}}(u))|=2$. Since $r\geq1$, it follows that $u$ receives at least $\min\{d_{G}(u), r\}$ colors in its neighborhood. If $u\not\in N(G)$, then $u\in T_{v_{i}}$ and $d_{G}(u)=d_{T_{v_{i}}}(u)$. Since $c_{i}$ is an $(L,r)$-coloring of $T_{v_{i}}$, it follows that vertex $u$ receives at least $\min\{d_{G}(u), r\}$ colors in its neighborhood. Hence, the coloring $c$ of $G$ also satisfies condition $(C2)$.
Thus $c$ is indeed an $(L,r)$-coloring of the graph $G$. Hence $\chi_{L,r}(G)\leq l$.
$\hfill\square$

\begin{lemma}\label{claim2'}
 Let $t$ be an integer with $t\geq \max\{5, k\}$. If $G\in \mathcal{G}_{n}$, then $\chi_{L,r}(G)\leq t$.
\end{lemma}

\noindent\textbf{Proof}. By Theorem \ref{theorem1.2'}, $\chi_{L,r}(C_{n})\leq5$. Since $t\geq5$, for any $t$-list $L$ of $G$, there exists an $(L,r)$-coloring $c_{0}$ of $C_{n}$ such that $|c_{0}(N_{C_{n}}(v_{i}))|=2$ for every vertex $v_{i}\in M(G)$. Moreover, since $t\geq k$, for any $v_{i}\in M(G)$, by $(i)$ of Lemma \ref{Lemma3.1'}, there exists an $(L,r)$-coloring $c_{i}$ of $T_{v_{i}}$ such that $c_{i}(v_{i})=c_{0}(v_{i})$, $c_{i}(v_{i-1})=c_{0}(v_{i-1})$ and $c_{i}(v_{i+1})=c_{0}(v_{i+1})$.
As $c_{0}$ and $c_{i}$ agree on these vertices for each $i$, we apply \text{(}\ref{eq:8'}\text{)} to obtain an $(L, r)$-coloring $c$ of $G$. By Lemma \ref{3.1-3.2} with $|L|=t$, we have $\chi_{L,r}(G)\leq t$, and so Lemma \ref{claim2'} is proved.
$\hfill\square$

\begin{theorem}\label{Th3.5'}
If $G\in \mathcal{G}_{5}$ and $r\geq2$, then $\chi_{L,r}(G)\in \{k,k+1\}$, and each of the following holds.\\
$(i)$ $\chi_{L,2}(G)=k+1$ if and only if for every vertex $v_{i}\in M(G)$, $d_{G}(v_{i-1})=d_{G}(v_{i+1})=2$;\\
$(ii)$ Suppose that $r\geq3$. Then $\chi_{L,r}(G)=k+1$ if and only if one of the following holds.\\
\indent $(ii$-$1)$ $\Delta(G)=3$;\\
\indent $(ii$-$2)$ $r=3$ and $d_{G}(v_{i})\leq3$ for every vertex $v_{i}\in V(C_{5})$.
\end{theorem}

\noindent\textbf{Proof}. For any $k$-list $L$ of $G$ and $M(G)=\{v~|~v\in V(C_{5})$ and $d_{G}(v)\geq3\}$, we assume that $C_{5}=v_{1}v_{2}\cdots v_{5}v_{1}$. We consider the following cases.

\noindent\textbf{Case 1:} $r=2$.

\par Since $\Delta(G)\geq3$ and $r=2$, we have $k+1=\min\{{r, \Delta(G)}\}+1+1=4$. For every vertex $v_{i}\in M(G)$, we have $d_{G}(v_{i})\geq3$. Since $r=2$, $d_{G}(v_{i})\geq3>2=r$. Without loss of generality assume that $v_{1}\in M(G)$. For any $4$-list $L$ of $G$, we obtain the coloring $c$ of $G$ as follows: As $|L(v_{i})|=4$, the vertices of $C_{5}$ can be colored by setting.

\begin{equation}\label{eqa}
\left\{
\begin{array}{ll}
c_{0}(v_{1})\in L(v_{1});\\
c_{0}(v_{2})\in L(v_{2})-\{c_{0}(v_{1})\};\\
c_{0}(v_{3})\in L(v_{3})-\{c_{0}(v_{1}), c_{0}(v_{2})\};\\
c_{0}(v_{4})\in L(v_{4})-\{c_{0}(v_{1}), c_{0}(v_{2}), c_{0}(v_{3})\};\\
c_{0}(v_{5})\in L(v_{5})-\{c_{0}(v_{1}), c_{0}(v_{3}), c_{0}(v_{4})\}.
\end{array}
\right.
\end{equation}

\par Since $|L(v_{4})-\{c_{0}(v_{1}), c_{0}(v_{2}), c_{0}(v_{3})\}|\geq1$ and $|L(v_{5})-\{c_{0}(v_{1}), c_{0}(v_{3}), c_{0}(v_{4})\}|\geq1$, it is possible to define $c_{0}(v_{4})\in L(v_{4})-\{c_{0}(v_{1}), c_{0}(v_{2}), c_{0}(v_{3})\}$ and $c_{0}(v_{5})\in L(v_{5})-\{c_{0}(v_{1}), c_{0}(v_{3}), c_{0}(v_{4})\}$.

\par For any $4$-list $L$ of $G$, when the above operations are finished, we obtain an $(L, 1)$-coloring $c_{0}$ of $C_{5}$ such that $1\leq |c_{0}(N_{C_{5}}(v_{1}))|\leq2$ and $|c_{0}(N_{C_{5}}(v_{i}))|=2$ $(i\neq1)$ for every vertex $v_{i}\in M(G)$. For any $v_{i}\in M(G)$, by Lemma \ref{Lemma3.1'}, there exists an $(L,2)$-coloring $c_{i}$ of $T_{v_{i}}$ such that $c_{i}(v_{i})=c_{0}(v_{i})$, $c_{i}(v_{i-1})=c_{0}(v_{i-1})$ and $c_{i}(v_{i+1})=c_{0}(v_{i+1})$.

\par As for each $i$, $c_{0}$ and $c_{i}$ agree on $\{v_{i-1}, v_{i}, v_{i+1}\}$, we apply \text{(}\ref{eq:8'}\text{)} to obtain an $L$-coloring $c$ of $G$. By Lemma \ref{3.1-3.2} with $|L|=4$, we have $\chi_{L,r}(G)\leq 4$. By \text{(}\ref{eq:2'}\text{)}, $\chi_{L,r}(G)\geq 3$. Thus $\chi_{L,r}(G)\in \{3,4\}=\{k, k+1\}$. We now prove the following necessary and sufficient condition for $\chi_{L,r}(G)=k+1$.

\noindent\textbf{Case 1.1:} For every $v_{i}\in M(G)$, $d_{G}(v_{i-1})=d_{G}(v_{i+1})=2$.

\begin{claim}\label{claim3.5'}
If $G\in \mathcal{G}_{5}$, for any $v_{i}\in M(G)$, $d_{G}(v_{i-1})=d_{G}(v_{i+1})=2$, then $\chi_{2}(G)\geq 4$.
\end{claim}

\par By contradiction, we assume $\chi_{2}(G)\leq 3$. Therefore, there exists a $(3, 2)$-coloring $c$ of graph $G$. Since $n=5$, for any $v_{i}\in M(G)$ and $d_{G}(v_{i-1})=d_{G}(v_{i+1})=2$, there must be $|M(G)|\leq \alpha(C_{5})=2$.

\par When $|M(G)|=1$, without loss of generality assume that $v_{1}\in M(G)$. Since $r=2$, we have $c(v_{2})$, $c(v_{3})$ and $c(v_{4})$ are mutually distinct. As $c(v_{1})\notin \{c(v_{2}), c(v_{3}), c(v_{4}), c(v_{5})\}$, there is no available color for vertex $v_{1}$. This contradicts the fact that $c$ is a $(3, 2)$-coloring. Thus $\chi_{2}(G)\geq 4$.

\par When $|M(G)|=2$, without loss of generality assume that $M(G)=\{v_{1}, v_{4}\}$. Since $r=2$, we have $c(v_{1})$, $c(v_{2})$ and $c(v_{3})$ are mutually distinct. As $c(v_{4})\notin \{c(v_{1}), c(v_{2}), c(v_{3}), c(v_{5})\}$, there is no available color for vertex $v_{4}$. This contradicts the fact that $c$ is a $(3, 2)$-coloring. Thus $\chi_{2}(G)\geq 4$, and so Claim \ref{claim3.5'} is proved.

\par By Claim \ref{claim3.5'} and \text{(}\ref{eq:2'}\text{)}, $4\leq\chi_{2}(G)\leq\chi_{L,2}(G)$. Thus in Case $1.1$, $\chi_{L,2}(G)=k+1=4$.

\noindent\textbf{Case 1.2:} There exists a vertex $v_{i}\in M(G)$ such that $d_{G}(v_{i-1})\neq2$ or $d_{G}(v_{i+1})\neq2$, that is $v_{i-1}\in M(G)$ or $v_{i+1}\in M(G)$.

\par There must be $|M(G)|\geq2$ and there exist two adjacent vertices in $M(G)$. Without loss of generality assume that $\{v_{1}, v_{5}\}\subseteq M(G)$. For any $3$-list $L$ of $G$, we need to consider whether the lists of $L(v_{1})$, $L(v_{4})$ and $L(v_{5})$ are the same or not.

\noindent\textbf{Case 1.2.1:} $L(v_{1})\neq L(v_{5})$ or $L(v_{5})\neq L(v_{4})$.

\par For any $3$-list $L$ of $G$, we obtain the coloring $c$ of $G$ as follows: If $L(v_{1})\neq L(v_{5})$, then $L(v_{1})-L(v_{5})\neq\emptyset$. Hence there exists $c_{0}(v_{1})\in L(v_{1})-L(v_{5})$. Since $|L(v_{i})|=3$, the other vertices of $C_{5}$ can be colored by setting the following:

\[
\left\{
\begin{array}{ll}
c_{0}(v_{2})\in L(v_{2})-\{c_{0}(v_{1})\};\\
c_{0}(v_{3})\in L(v_{3})-\{c_{0}(v_{1}), c_{0}(v_{2})\};\\
c_{0}(v_{4})\in L(v_{4})-\{c_{0}(v_{2}), c_{0}(v_{3})\};\\
c_{0}(v_{5})\in L(v_{5})-\{c_{0}(v_{1}), c_{0}(v_{3}), c_{0}(v_{4})\}.
\end{array}
\right.
\]

\par Since $c_{0}(v_{1}) \notin L(v_{5})$, we have $|L(v_{5})-\{c_{0}(v_{1}), c_{0}(v_{3}), c_{0}(v_{4})\}|\geq1$. Hence it is possible to define $c_{0}(v_{5})\in L(v_{5})-\{c_{0}(v_{1}), c_{0}(v_{3}), c_{0}(v_{4})\}$.

\par For any $3$-list $L$ of $G$, when the above operations are finished, we obtain an $(L, 1)$-coloring $c_{0}$ of $C_{5}$ such that $1\leq |c_{0}(N_{C_{5}}(v_{1}))|\leq2$, $1\leq |c_{0}(N_{C_{5}}(v_{5}))|\leq2$ and $|c_{0}(N_{C_{5}}(v_{i}))|=2$ $(i\not\in\{1,5\})$.

\par If $L(v_{4})\neq L(v_{5})$, then $L(v_{4})-L(v_{5})\neq\emptyset$. Hence there exists $c_{0}(v_{4})\in L(v_{4})-L(v_{5})$. Since $|L(v_{i})|=3$, the other vertices of $C_{5}$ can be colored by choosing the following:

\[
\left\{
\begin{array}{ll}
c_{0}(v_{3})\in L(v_{3})-\{c_{0}(v_{4})\};\\
c_{0}(v_{2})\in L(v_{2})-\{c_{0}(v_{3}), c_{0}(v_{4})\};\\
c_{0}(v_{1})\in L(v_{1})-\{c_{0}(v_{2}), c_{0}(v_{3})\};\\
c_{0}(v_{5})\in L(v_{5})-\{c_{0}(v_{1}), c_{0}(v_{3}), c_{0}(v_{4})\}.
\end{array}
\right.
\]

\par Since $c_{0}(v_{4}) \notin L(v_{5})$, we have $|L(v_{5})-\{c_{0}(v_{1}), c_{0}(v_{3}), c_{0}(v_{4})\}|\geq1$. Hence it is possible to define $c_{0}(v_{5})\in L(v_{5})-\{c_{0}(v_{1}), c_{0}(v_{3}), c_{0}(v_{4})\}$.

\par For any $3$-list $L$ of $G$, when the above operations are finished, we obtain an $(L, 1)$-coloring $c_{0}$ of $C_{5}$ such that $1\leq |c_{0}(N_{C_{5}}(v_{5}))|\leq2$, $1\leq |c_{0}(N_{C_{5}}(v_{1}))|\leq2$ and $|c_{0}(N_{C_{5}}(v_{i}))|=2$ $(i\not\in\{1,5\})$.

\par For any $v_{i}\in M(G)$, by Lemma \ref{Lemma3.1'}, there exists an $(L,2)$-coloring $c_{i}$ of $T_{v_{i}}$ such that $c_{i}(v_{i})=c_{0}(v_{i})$, $c_{i}(v_{i-1})=c_{0}(v_{i-1})$ and $c_{i}(v_{i+1})=c_{0}(v_{i+1})$.
Apply \text{(}\ref{eq:8'}\text{)} to form a coloring $c$ of $G$. By Lemma \ref{3.1-3.2} with $|L|=3$, we conclude that in Case $1.2.1$, $\chi_{L,r}(G)=\chi_{L,2}(G)\leq 3$.

\noindent\textbf{Case 1.2.2:} $L(v_{1})=L(v_{4})=L(v_{5})$.

\par For any $3$-list $L$ of $G$, we obtain the coloring $c$ of $G$ as follows: As $|L(v_{i})|=3$, there exists $c_{0}(v_{1})\in L(v_{1})$. Since $L(v_{1})=L(v_{4})=L(v_{5})$, the other vertices of $C_{5}$ can be colored by making these selections:

\[
\left\{
\begin{array}{ll}
c_{0}(v_{2})\in L(v_{2})-\{c_{0}(v_{1})\};\\
c_{0}(v_{3})\in L(v_{3})-\{c_{0}(v_{1}), c_{0}(v_{2})\};\\
c_{0}(v_{4})=c_{0}(v_{1});\\
c_{0}(v_{5})\in L(v_{5})-\{c_{0}(v_{1}), c_{0}(v_{3}), c_{0}(v_{4})\}.
\end{array}
\right.
\]

\par Since $c_{0}(v_{4})=c_{0}(v_{1})$, we have $|L(v_{5})-\{c_{0}(v_{1}), c_{0}(v_{3}), c_{0}(v_{4})\}|\geq1$. Hence it is possible to define $c_{0}(v_{5})\in L(v_{5})-\{c_{0}(v_{1}), c_{0}(v_{3}), c_{0}(v_{4})\}$.

\par For any $3$-list $L$ of $G$, when the above operations are finished, we obtain an $(L, 1)$-coloring $c_{0}$ of $C_{5}$ such that $|c_{0}(N_{C_{5}}(v_{5}))|=1$, $1\leq|c_{0}(N_{C_{5}}(v_{1}))|\leq2$ and $|c_{0}(N_{C_{5}}(v_{i}))|=2$ $(i\not\in\{1, 5\})$. For any $v_{i}\in M(G)$, by Lemma \ref{Lemma3.1'}, there exists an $(L,2)$-coloring $c_{i}$ of $T_{v_{i}}$ such that $c_{i}(v_{i})=c_{0}(v_{i})$, $c_{i}(v_{i-1})=c_{0}(v_{i-1})$ and $c_{i}(v_{i+1})=c_{0}(v_{i+1})$.
As $c_{0}$ and $c_{i}$ agree on these vertices for each $i$, we apply \text{(}\ref{eq:8'}\text{)} to obtain an $(L, 2)$-coloring $c$ of $G$. By Lemma \ref{3.1-3.2} with $|L|=3$, we have $\chi_{L,r}(G)=\chi_{L,2}(G)\leq 3$.

\par By \text{(}\ref{eq:2'}\text{)}, $\chi_{L,2}(G)\geq 3$. Thus in Case $1.2.2$, $\chi_{L,r}(G)=\chi_{L,2}(G)=3\neq4$.

\noindent\textbf{Case 2:} $r\geq3$.

\par As $r\geq3$ and $\Delta(G)\geq3$, $k+1=\min\{{r, \Delta(G)}\}+1+1\geq5$. By Lemma \ref{claim2'}, then $\chi_{L,r}(G)\leq k+1$.

\par By \text{(}\ref{eq:2'}\text{)}, $\chi_{L,r}(G)\geq k$. Thus $\chi_{L,r}(G)\in \{k, k+1\}$. We now prove the following necessary and sufficient condition for $\chi_{L,r}(G)=k+1$.

\begin{claim}\label{claim3'}
If $\Delta(G)=3$ or if $r=3$ and $d_{G}(v_{i})\leq3$ for every vertex $v_{i}\in V(C_{5})$, then $\chi_{L,r}(G)=k+1$.
\end{claim}

\par If $\Delta(G)=3$ or $r=3$, then $k+1=\min\{{r, \Delta(G)}\}+1+1=5$. By Lemma \ref{claim2'}, $\chi_{L,r}(G)\leq k+1$. Next, we only need to prove that $\chi_{L, r}(G)\geq k+1$.

\par When $\Delta(G)=3$, for every $v_{i}\in M(G)$, $d_{G}(v_{i})\geq3$ and $d_{G}(v_{i})\leq\Delta(G)=3$, it follows that $d_{G}(v_{i})=3$. As $r\geq3$ and $C_{5}\subseteq G$, we obtain $\chi_{L,r}(G)\geq\chi_{r}(G)\geq\chi_{3}(G)=\chi_{\Delta}(G)=\chi(G^{2})\geq\chi(C_{5}^{2})=\chi(K_{5})=5$. Thus $\chi_{L,r}(G)\geq k+1$.

\par When $r=3$ and $d_{G}(v_{i})\leq3$ for every vertex $v_{i}\in V(C_{5})$, there must be $\chi_{3}(G)\geq5$. Otherwise, we assume $\chi_{3}(G)\leq 4$. Therefore, there exists a $(4, 3)$-coloring $c$ of graph $G$. Since $r=3$ and $d_{G}(v_{i})\leq3$ for any $v_{i}\in V(C_{5})$, $c(v_{1})$, $c(v_{2})$, $c(v_{3})$ and $c(v_{4})$ are mutually distinct. Since $c(v_{5})\notin \{c(v_{1}), c(v_{2}), c(v_{3}), c(v_{4})\}$, there is no available color for vertex $v_{5}$. This contradicts the fact that $c$ is a $(4, 3)$-coloring. Thus $\chi_{3}(G)\geq 5$. By \text{(}\ref{eq:2'}\text{)}, $\chi_{L,3}(G)\geq\chi_{3}(G)\geq5$. Thus $\chi_{L,3}(G)\geq k+1$, and so Claim \ref{claim3'} is proved.

\par Claim \ref{claim3'} implies the sufficiency of Theorem \ref{Th3.5'} $(ii)$.

\begin{claim}\label{claim4'}
If $\min\{\Delta(G), r\}>3$, or if $\Delta(G)>3$ and for some vertex $v_{i}\in V(C_{5})$, $d_{G}(v_{i})\geq4$, then $\chi_{L,r}(G)=k$.
\end{claim}

\par Assume first that $\min\{\Delta(G), r\}>3$. Then $k\geq5$. Therefore, by \text{(}\ref{eq:2'}\text{)} and Lemma \ref{claim2'}, we have $\chi_{L,r}(G)=k$. Hence we assume that $\Delta(G)>3$ and for some vertex $v_{i}\in V(C_{5})$, $d_{G}(v_{i})\geq4$.
We only need to consider the case $d_{G}(v_{i})\geq4$ and $r=3$. Therefore, we have $d_{G}(v_{i})>r$ and $k=\min\{{r, \Delta(G)}\}+1=4$. Without loss of generality, we assume that $d_{G}(v_{1})\geq4>3=r$.

\par For any $k$-list $L$ of $G$, by (\ref{eqa}), we obtain an $(L, 1)$-coloring $c_{0}$ of $C_{5}$ such that $1\leq|c_{0}(N_{C_{5}}(v_{1}))|\leq2$ and $|c_{0}(N_{C_{5}}(v_{i}))|=2$ $(i\neq1)$ for every vertex $v_{i}\in M(G)$. For $v_{1}\in M(G)$, by Lemma \ref{Lemma3.1'}, there exists an $(L,r)$-coloring $c_{1}$ of $T_{v_{1}}$ such that $c_{1}(v_{1})=c_{0}(v_{1})$, $c_{1}(v_{5})=c_{0}(v_{5})$ and $c_{1}(v_{2})=c_{0}(v_{2})$. Similarly, for every other vertex $v_{i}$ in $M(G)$, by $(i)$ of Lemma \ref{Lemma3.1'}, there exists an $(L,r)$-coloring $c_{i}$ of $T_{v_{i}}$ such that $c_{i}(v_{i})=c_{0}(v_{i})$, $c_{i}(v_{i-1})=c_{0}(v_{i-1})$ and $c_{i}(v_{i+1})=c_{0}(v_{i+1})$.

\par As for each $i$, $c_{0}$ and $c_{i}$ agree on $\{v_{i-1}, v_{i}, v_{i+1}\}$, we apply \text{(}\ref{eq:8'}\text{)} to obtain an $(L, r)$-coloring $c$ of $G$. By Lemma \ref{3.1-3.2} with $|L|=k$, we have $\chi_{L,r}(G)\leq k$. By \text{(}\ref{eq:2'}\text{)}, $\chi_{L,r}(G)\geq k$. Thus $\chi_{L,r}(G)=k$, and so Claim \ref{claim4'} is proved.

\par Claim \ref{claim4'} implies the necessity of Theorem \ref{Th3.5'} $(ii)$.
$\hfill\square$

\begin{theorem}\label{Theorem3.4}
If $G\in \mathcal{G}_{n}$, $n\neq5$, then
$$
\chi_{L,2}(G)\in
\begin{cases}
\{3\} & n\equiv 0\pmod3; \\
\{3, 4\} & n\not\equiv 0\pmod3. \\
\end{cases}
$$
\end{theorem}

\noindent\textbf{Proof}. We consider the following cases.

\noindent\textbf{Case 1:} $n\equiv 0\pmod3$.

\par By \text{(}\ref{eq:2'}\text{)}, it follows that $\chi_{L, 2}(G)\geq 3$. Next, we only need to prove that $\chi_{L, 2}(G)\leq 3$.

\par Since $\Delta(G)\geq3$ and $r=2$, we have $k=\min\{{r, \Delta(G)}\}+1=3$. By Theorem \ref{theorem1.2'}, $\chi_{L,r}(C_{n})=3$. For any $3$-list $L$ of $G$, there exists an $(L, 2)$-coloring $c_{0}$ of $C_{n}$ such that $|c_{0}(N_{C_{n}}(v_{i}))|=2$ for every vertex $v_{i}\in M(G)$. For any $v_{i}\in M(G)$, by $(i)$ of Lemma \ref{Lemma3.1'}, there exists an $(L,2)$-coloring $c_{i}$ of $T_{v_{i}}$ such that $c_{i}(v_{i})=c_{0}(v_{i})$, $c_{i}(v_{i-1})=c_{0}(v_{i-1})$ and $c_{i}(v_{i+1})=c_{0}(v_{i+1})$.
As $c_{0}$ and $c_{i}$ agree on these vertices for each $i$, we apply \text{(}\ref{eq:8'}\text{)} to obtain an $(L, 2)$-coloring $c$ of $G$. By Lemma \ref{3.1-3.2} with $|L|=3$, we have $\chi_{L,2}(G)\leq 3$. Therefore, we have $\chi_{L,2}(G)=3$.

\noindent\textbf{Case 2:} $n\not\equiv 0\pmod3$.

\par By \text{(}\ref{eq:2'}\text{)}, it follows that $\chi_{L, 2}(G)\geq 3$. Since $\Delta(G)\geq3$ and $r=2$, we have $k+1=\min\{{r, \Delta(G)}\}+1+1=4$. Next, we only need to prove that $\chi_{L, 2}(G)\leq 4$.

\par By Theorem \ref{theorem1.2'}, we have $\chi_{L,2}(C_{n})=4=k+1$. For any $4$-list $L$ of $G$, there exists an $(L, 2)$-coloring $c_{0}$ of $C_{n}$ such that $|c_{0}(N_{C_{n}}(v_{i}))|=2$ for every vertex $v_{i}\in M(G)$. Therefore, for any $v_{i}\in M(G)$, by $(i)$ of Lemma \ref{Lemma3.1'}, there exists an $(L,2)$-coloring $c_{i}$ of $T_{v_{i}}$ such that $c_{i}(v_{i})=c_{0}(v_{i})$, $c_{i}(v_{i-1})=c_{0}(v_{i-1})$ and $c_{i}(v_{i+1})=c_{0}(v_{i+1})$.
As for each $i$, $c_{0}$ and $c_{i}$ agree on $\{v_{i-1}, v_{i}, v_{i+1}\}$, we apply \text{(}\ref{eq:8'}\text{)} to obtain an $(L, 2)$-coloring $c$ of $G$. By Lemma \ref{3.1-3.2} with $|L|=4$, we have $\chi_{L,2}(G)\leq 4$. Therefore, we have $\chi_{L,2}(G)\in \{3, 4\}$.

\par This completes the proof of Theorem \ref{Theorem3.4}.
$\hfill\square$

\par We have the following observations (\ref{at-1}) and (\ref{at-2}).
\begin{equation}\label{at-1}
\mbox{If $G\in \mathcal{G}_{n}$, $n\neq5$, $n\not\equiv0$ (mod 3) and $|M(G)|=1$, then $\chi_{L,2}(G)=4$.}
\end{equation}

\par Since $r=2$ and $n\not\equiv 0\pmod3$, we must have $\chi_{2}(G)\geq4$. Otherwise, we assume $\chi_{2}(G)\leq 3$. Therefore, there exists a $(3, 2)$-coloring $c$ : $V(G)\rightarrow \{1, 2, 3\}$. Without loss of generality assume that $v_{1}\in M(G)$, $c(v_{1})=1$ and $c(v_{2})=2$. Since $r=2$, for any $v_{i}\in V(C_{n})~(3\leq i\leq n-2)$, if $i\equiv 1\pmod3$, then $c(v_{i})=1$; if $i\equiv 2\pmod3$, then $c(v_{i})=2$; if $i\equiv 0\pmod3$, then $c(v_{i})=3$.

\par If $n\equiv 1\pmod3$, then $c(v_{n-1})=3$. Therefore, $c(v_{1})$, $c(v_{n-1})$ and $c(v_{n-2})$ are mutually distinct. Since $c(v_{n})\notin \{c(v_{1}), c(v_{n-1}), c(v_{n-2})\}$, there is no available color for vertex $v_{n}$; when $n\equiv 2\pmod3$, we have $c(v_{1})$, $c(v_{n-2})$ and $c(v_{n-3})$ are mutually distinct. Since $c(v_{n-1})\notin \{c(v_{1}), c(v_{n-2}), c(v_{n-3})\}$, there is no available color for vertex $v_{n-1}$. This contradicts the fact that $c$ is a $(3, 2)$-coloring. Thus $\chi_{2}(G)\geq 4$. By \text{(}\ref{eq:2'}\text{)} and Theorem \ref{Theorem3.4}, $4\leq\chi_{2}(G)\leq\chi_{L,2}(G)\leq4$. Thus $\chi_{L,2}(G)=4$. This proves (\ref{at-1}).

\begin{equation}\label{at-2}
\mbox{If $G\in \mathcal{G}_{n}$, $n\neq5$, $n\not\equiv0$ (mod 3) and $|M(G)|\geq n-1$, then $\chi_{L,2}(G)=3$.}
\end{equation}

\par Since $|M(G)|\geq n-1$, without loss of generality assume that $\{v_{1}, v_{2},\cdots,v_{n-1}\}\subseteq M(G)$. For any $3$-list $L$ of $G$, we obtain the coloring $c$ of $G$ as follows: As $|L(v_{i})|=3$ and $|M(G)|\geq n-1$, there exists $c_{0}(v_{1})\in L(v_{1})$ and $c_{0}(v_{i})\in L(v_{i})-\{c_{0}(v_{i-1})\}$ for $2\leq i \leq n-2$. The other vertices of $C_{n}$ can be colored by setting the following:

\[
\left\{
\begin{array}{ll}
c_{0}(v_{n-1})\in L(v_{n-1})-\{c_{0}(v_{1}), c_{0}(v_{n-2})\};\\
c_{0}(v_{n})\in L(v_{n})-\{c_{0}(v_{1}), c_{0}(v_{n-1})\}.
\end{array}
\right.
\]

\par Since $|L(v_{n-1})-\{c_{0}(v_{1}), c_{0}(v_{n-2})\}|\geq1$ and $|L(v_{n})-\{c_{0}(v_{1}), c_{0}(v_{n-1})\}|\geq1$, it is possible to define $c_{0}(v_{n-1})\in L(v_{n-1})-\{c_{0}(v_{1}), c_{0}(v_{n-2})\}$ and $c_{0}(v_{n})\in L(v_{n})-\{c_{0}(v_{1}), c_{0}(v_{n-1})\}$.

\par For any $3$-list $L$ of $G$, when the above operations are finished, we obtain an $(L, 1)$-coloring $c_{0}$ of $C_{n}$ such that $|c_{0}(N_{C_{n}}(v_{n}))|=2$ and $1\leq|c_{0}(N_{C_{n}}(v_{i}))|\leq2$ $(1\leq i\leq n-1)$ for every vertex $v_{i}\in M(G)$. Since $v_{i}\in M(G)$ $(i\neq n)$, $d_{G}(v_{i})\geq3>2=r$. For any $v_{i}\in M(G)$, by Lemma \ref{Lemma3.1'}, there exists an $(L,2)$-coloring $c_{i}$ of $T_{v_{i}}$ such that $c_{i}(v_{i})=c_{0}(v_{i})$, $c_{i}(v_{i-1})=c_{0}(v_{i-1})$ and $c_{i}(v_{i+1})=c_{0}(v_{i+1})$.

\par As $c_{0}$ and $c_{i}$ agree on these vertices for each $i$, we apply \text{(}\ref{eq:8'}\text{)} to obtain an $(L, 2)$-coloring $c$ of $G$. By Lemma \ref{3.1-3.2} with $|L|=3$, we have $\chi_{L,2}(G)\leq 3$. By \text{(}\ref{eq:2'}\text{)}, $\chi_{L,2}(G)\geq3$. Thus $\chi_{L,2}(G)=3$. This proves (\ref{at-2}).

\par If $G\in \mathcal{G}_{n}$ and $n\neq5$, then either $\chi_{L,2}(G)=3$ or $\chi_{L,2}(G)=4$ is possible. If $n\not\equiv0$ (mod 3), we present the values of $\chi_{L,2}(G)$ for $|M(G)|=1$ and $|M(G)|\geq n-1$ respectively. When $2\leq|M(G)|\leq n-2$, several examples are presented to show that both values are attainable (see Table \ref{tab:my_table}).

\begin{table}
\centering
\begin{tikzpicture}[x=1.00mm, y=1.00mm, inner xsep=0pt, inner ysep=0pt, outer xsep=0pt, outer ysep=0pt]
\path[line width=0mm] (90.78,18.18) rectangle +(127.63,56.41);
\definecolor{L}{rgb}{0,0,0}
\path[line width=0.30mm, draw=L] (92.78,72.48) -- (216.27,72.41);
\path[line width=0.30mm, draw=L] (92.78,60.12) -- (216.14,59.80);
\path[line width=0.30mm, draw=L] (92.78,40.13) -- (216.14,39.94);
\path[line width=0.30mm, draw=L] (93.07,20.62) -- (216.41,20.35);
\path[line width=0.30mm, draw=L] (125.13,72.48) -- (125.34,20.49);
\path[line width=0.30mm, draw=L] (163.88,72.25) -- (163.86,20.18);
\path[line width=0.30mm, draw=L] (216.32,72.48) -- (216.39,20.38);
\path[line width=0.30mm, draw=L] (92.89,72.59) -- (93.20,20.35);
\path[line width=0.30mm, draw=L] (92.78,72.48) -- (124.94,60.26);
\draw(117.12,66.47) node[anchor=base west]{\fontsize{14.23}{17.07}\selectfont n};
\draw(130,64.41) node[anchor=base west]{\fontsize{14.23}{17.07}\selectfont $n\equiv1$ (mod 3)};
\draw(168,64) node[anchor=base west]{\fontsize{14.23}{17.07}\selectfont $n\neq5$, $n\equiv2$ (mod 3)};
\draw(107.69,49.04) node[anchor=base west]{\fontsize{14.23}{17.07}\selectfont 3};
\draw(107,28.87) node[anchor=base west]{\fontsize{14.23}{17.07}\selectfont 4};
\draw(93.68,62.48) node[anchor=base west]{\fontsize{14.23}{17.07}\selectfont $\chi_{L,2}(G)$};
\definecolor{F}{rgb}{1,1,1}
\path[line width=0.30mm, draw=L, fill=F] (145.09,54.93) circle (0.00mm);
\path[line width=0.15mm, draw=L, fill=F] (145.03,54.99) circle (0.50mm);
\path[line width=0.15mm, draw=L, fill=F] (145.03,54.99) circle (0.50mm);
\path[line width=0.15mm, draw=L, fill=F] (140.04,49.95) circle (0.50mm);
\path[line width=0.15mm, draw=L, fill=F] (150.19,49.89) circle (0.50mm);
\path[line width=0.15mm, draw=L, fill=F] (145.15,45.03) circle (0.50mm);
\path[line width=0.15mm, draw=L, fill=F] (157.45,50.01) circle (0.50mm);
\path[line width=0.15mm, draw=L, fill=F] (132.54,49.95) circle (0.50mm);
\path[line width=0.15mm, draw=L] (145.39,54.93) -- (150.13,50.31);
\path[line width=0.15mm, draw=L] (147.49,53.25);
\path[line width=0.15mm, draw=L] (144.73,54.99) -- (140.29,50.31);
\path[line width=0.15mm, draw=L] (140.23,49.47) -- (144.85,45.15);
\path[line width=0.15mm, draw=L] (149.83,49.77) -- (145.51,45.27);
\path[line width=0.15mm, draw=L] (150.73,50.01) -- (157.03,49.95);
\path[line width=0.15mm, draw=L] (133.02,49.95) -- (139.57,49.95);
\path[line width=0.15mm, draw=L, fill=F] (145.03,54.99) circle (0.50mm);
\path[line width=0.15mm, draw=L, fill=F] (182.26,26.37) circle (0.50mm);
\path[line width=0.15mm, draw=L, fill=F] (145.03,54.99) circle (0.50mm);
\path[line width=0.15mm, draw=L, fill=F] (196.52,34.54) circle (0.50mm);
\path[line width=0.15mm, draw=L, fill=F] (197.32,30.42) circle (0.50mm);
\path[line width=0.15mm, draw=L, fill=F] (192.38,26.39) circle (0.50mm);
\path[line width=0.15mm, draw=L, fill=F] (187.20,26.38) circle (0.50mm);
\path[line width=0.30mm, draw=L, fill=F] (145.09,54.93) circle (0.00mm);
\path[line width=0.15mm, draw=L, fill=F] (145.03,54.99) circle (0.50mm);
\path[line width=0.15mm, draw=L, fill=F] (145.03,54.99) circle (0.50mm);
\path[line width=0.15mm, draw=L, fill=F] (140.04,49.95) circle (0.50mm);
\path[line width=0.15mm, draw=L, fill=F] (150.19,49.89) circle (0.50mm);
\path[line width=0.15mm, draw=L, fill=F] (145.15,45.03) circle (0.50mm);
\path[line width=0.15mm, draw=L, fill=F] (157.45,50.01) circle (0.50mm);
\path[line width=0.15mm, draw=L] (145.39,54.93) -- (150.13,50.31);
\path[line width=0.15mm, draw=L] (147.49,53.25);
\path[line width=0.15mm, draw=L] (144.73,54.99) -- (140.29,50.31);
\path[line width=0.15mm, draw=L] (140.23,49.47) -- (144.85,45.15);
\path[line width=0.15mm, draw=L] (149.83,49.77) -- (145.51,45.27);
\path[line width=0.15mm, draw=L] (150.73,50.01) -- (157.03,49.95);
\path[line width=0.15mm, draw=L, fill=F] (145.03,54.99) circle (0.50mm);
\path[line width=0.15mm, draw=L, fill=F] (145.03,54.99) circle (0.50mm);
\path[line width=0.30mm, draw=L, fill=F] (145.09,54.93) circle (0.00mm);
\path[line width=0.15mm, draw=L, fill=F] (145.03,54.99) circle (0.50mm);
\path[line width=0.15mm, draw=L, fill=F] (145.03,54.99) circle (0.50mm);
\path[line width=0.15mm, draw=L, fill=F] (140.04,49.95) circle (0.50mm);
\path[line width=0.15mm, draw=L, fill=F] (150.19,49.89) circle (0.50mm);
\path[line width=0.15mm, draw=L, fill=F] (145.15,45.03) circle (0.50mm);
\path[line width=0.15mm, draw=L, fill=F] (157.45,50.01) circle (0.50mm);
\path[line width=0.15mm, draw=L] (145.39,54.93) -- (150.13,50.31);
\path[line width=0.15mm, draw=L] (147.49,53.25);
\path[line width=0.15mm, draw=L] (144.73,54.99) -- (140.29,50.31);
\path[line width=0.15mm, draw=L] (140.23,49.47) -- (144.85,45.15);
\path[line width=0.15mm, draw=L] (149.83,49.77) -- (145.51,45.27);
\path[line width=0.15mm, draw=L] (150.73,50.01) -- (157.03,49.95);
\path[line width=0.15mm, draw=L, fill=F] (145.03,54.99) circle (0.50mm);
\path[line width=0.15mm, draw=L, fill=F] (145.03,54.99) circle (0.50mm);
\path[line width=0.15mm, draw=L] (182.78,26.35) -- (186.83,26.39);
\path[line width=0.15mm, draw=L] (187.84,26.33) -- (191.89,26.37);
\path[line width=0.15mm, draw=L] (193.00,26.42) -- (197.00,30.06);
\path[line width=0.15mm, draw=L] (192.91,34.48) -- (196.09,34.48);
\path[line width=0.15mm, draw=L, fill=F] (192.37,34.71) circle (0.50mm);
\path[line width=0.15mm, draw=L, fill=F] (187.24,34.65) circle (0.50mm);
\path[line width=0.15mm, draw=L, fill=F] (201.50,49.98) circle (0.50mm);
\path[line width=0.15mm, draw=L, fill=F] (182.18,34.48) circle (0.50mm);
\path[line width=0.15mm, draw=L] (182.74,34.54) -- (186.79,34.58);
\path[line width=0.15mm, draw=L] (187.85,34.60) -- (191.90,34.64);
\path[line width=0.15mm, draw=L] (192.78,34.52) -- (196.88,30.76);
\path[line width=0.15mm, draw=L] (197.89,49.96) -- (201.07,49.96);
\path[line width=0.15mm, draw=L, fill=F] (135.46,26.37) circle (0.50mm);
\path[line width=0.15mm, draw=L, fill=F] (154.59,30.44) circle (0.50mm);
\path[line width=0.15mm, draw=L, fill=F] (150.53,30.42) circle (0.50mm);
\path[line width=0.15mm, draw=L, fill=F] (145.59,26.39) circle (0.50mm);
\path[line width=0.15mm, draw=L, fill=F] (140.40,26.38) circle (0.50mm);
\path[line width=0.15mm, draw=L] (135.98,26.35) -- (140.03,26.39);
\path[line width=0.15mm, draw=L] (141.05,26.33) -- (145.10,26.37);
\path[line width=0.15mm, draw=L] (146.20,26.42) -- (150.31,30.02);
\path[line width=0.15mm, draw=L] (151.05,30.38) -- (154.22,30.38);
\path[line width=0.15mm, draw=L, fill=F] (145.58,34.71) circle (0.50mm);
\path[line width=0.15mm, draw=L, fill=F] (140.44,34.65) circle (0.50mm);
\path[line width=0.15mm, draw=L, fill=F] (135.39,34.48) circle (0.50mm);
\path[line width=0.15mm, draw=L] (135.94,34.54) -- (139.99,34.58);
\path[line width=0.15mm, draw=L] (141.05,34.60) -- (145.11,34.64);
\path[line width=0.15mm, draw=L] (145.98,34.52) -- (150.23,30.88);
\path[line width=0.15mm, draw=L, fill=F] (149.68,26.21) circle (0.50mm);
\path[line width=0.15mm, draw=L] (146.04,26.28) -- (149.22,26.28);
\path[line width=0.15mm, draw=L] (192.98,26.28) -- (196.16,26.28);
\path[line width=0.15mm, draw=L, fill=F] (196.73,26.27) circle (0.50mm);
\definecolor{F}{rgb}{0,0,0}
\path[line width=0.30mm, draw=L, fill=F] (135.22,32.15) circle (0.25mm);
\path[line width=0.30mm, draw=L, fill=F] (135.22,30.51) circle (0.25mm);
\path[line width=0.30mm, draw=L, fill=F] (135.22,28.88) circle (0.25mm);
\path[line width=0.30mm, draw=L, fill=F] (182.25,32.15) circle (0.25mm);
\path[line width=0.30mm, draw=L, fill=F] (182.25,30.51) circle (0.25mm);
\path[line width=0.30mm, draw=L, fill=F] (182.25,28.88) circle (0.25mm);
\definecolor{F}{rgb}{1,1,1}
\path[line width=0.15mm, draw=L, fill=F] (182.26,45.97) circle (0.50mm);
\path[line width=0.15mm, draw=L, fill=F] (196.52,54.14) circle (0.50mm);
\path[line width=0.15mm, draw=L, fill=F] (197.32,50.02) circle (0.50mm);
\path[line width=0.15mm, draw=L, fill=F] (192.38,45.99) circle (0.50mm);
\path[line width=0.15mm, draw=L, fill=F] (187.20,45.99) circle (0.50mm);
\path[line width=0.15mm, draw=L] (182.78,45.95) -- (186.83,45.99);
\path[line width=0.15mm, draw=L] (187.84,45.93) -- (191.89,45.97);
\path[line width=0.15mm, draw=L] (193.00,46.02) -- (197.00,49.67);
\path[line width=0.15mm, draw=L] (192.91,54.09) -- (196.09,54.09);
\path[line width=0.15mm, draw=L, fill=F] (192.37,54.31) circle (0.50mm);
\path[line width=0.15mm, draw=L, fill=F] (187.24,54.25) circle (0.50mm);
\path[line width=0.15mm, draw=L, fill=F] (182.18,54.08) circle (0.50mm);
\path[line width=0.15mm, draw=L] (182.74,54.14) -- (186.79,54.18);
\path[line width=0.15mm, draw=L] (187.85,54.20) -- (191.90,54.24);
\path[line width=0.15mm, draw=L] (192.78,54.13) -- (196.88,50.36);
\path[line width=0.15mm, draw=L] (192.98,45.89) -- (196.16,45.89);
\path[line width=0.15mm, draw=L, fill=F] (196.73,45.87) circle (0.50mm);
\definecolor{F}{rgb}{0,0,0}
\path[line width=0.30mm, draw=L, fill=F] (182.25,51.75) circle (0.25mm);
\path[line width=0.30mm, draw=L, fill=F] (182.25,50.12) circle (0.25mm);
\path[line width=0.30mm, draw=L, fill=F] (182.25,48.48) circle (0.25mm);
\end{tikzpicture}%
\caption{Showing values of $\chi_{L,2}(G)$}
\label{tab:my_table}
\end{table}

\begin{theorem}\label{Theorem3.2'}
If $G\in \mathcal{G}_{n}$, $n\neq5$ and $r\geq3$, then $\chi_{L,r}(G)=\min\{r,\Delta(G)\}+1=k$.
\end{theorem}

\noindent\textbf{Proof}. By \text{(}\ref{eq:2'}\text{)}, it follows that $\chi_{L,r}(G)\geq k$. Next, we only need to prove that $\chi_{L,r}(G)\leq k$.

\par Since $\Delta(G)\geq3$ and $r\geq3$, we have $k=\min\{{r, \Delta(G)}\}+1\geq 4$. By Theorem \ref{theorem1.2'}, $\chi_{L,r}(C_{n})\leq 4$. For any $k$-list $L$ of $G$, there exists an $(L, r)$-coloring $c_{0}$ of $C_{n}$ such that $|c_{0}(N_{C_{n}}(v_{i}))|=2$ for every vertex $v_{i}\in M(G)$. Therefore, for any $v_{i}\in M(G)$, by $(i)$ of Lemma \ref{Lemma3.1'}, there exists an $(L,r)$-coloring $c_{i}$ of $T_{v_{i}}$ such that $c_{i}(v_{i})=c_{0}(v_{i})$, $c_{i}(v_{i-1})=c_{0}(v_{i-1})$ and $c_{i}(v_{i+1})=c_{0}(v_{i+1})$.

\par As for each $i$, $c_{0}$ and $c_{i}$ agree on $\{v_{i-1}, v_{i}, v_{i+1}\}$, we apply \text{(}\ref{eq:8'}\text{)} to obtain an $(L, r)$-coloring $c$ of $G$. By Lemma \ref{3.1-3.2} with $|L|=k$, we have $\chi_{L,r}(G)\leq k$. Thus we have $\chi_{L,r}(G)=k=\min\{r,\Delta(G)\}+1$.
$\hfill\square$

\par In this paper, Theorem \ref{Theorem3.3'} proves all cases for $r=1$; Theorem \ref{Th3.5'} proves the case where $n=5$ and $r\geq 2$; Theorem \ref{Theorem3.4} proves the case where $n\neq 5$ and $r=2$; Theorem \ref{Theorem3.2'} proves the case where $n\neq 5$ and $r\geq 3$. Hence, by these four theorems, it is known that when $n\neq5$ and $r\geq3$, $\chi_{L,r}(G)=k$; when $n=5$ or $r\leq2$, $\chi_{L,r}(G)\in \{k, k+1\}$. This completes the proof of Theorem \ref{TH1.6'}. It shows that for the value of $\chi_{L,r}(G)$, both $k$ and $k+1$ can be actually obtained, and there are no other possible values. And, if $k\geq5$, then $\chi_{L,r}(G)=k$; $\chi_{L,r}(G)$ can take the value of $k+1$ only when $k\leq4$.\\

\noindent\textbf{Data Availability} This manuscript has no associated data.\\

\noindent\textbf{\LARGE Declarations}\\

\noindent\textbf{Competing Interests} The authors declare that they have no competing interests.


\begin{thebibliography}{111}


\bibitem{SMS13} S. Akbari, M. Ghanbari, S. Jahanbekam, {\em On the list dynamic coloring of graphs}, Discrete Appl. Math. 157 (14) (2009) 3005-3007.

\bibitem{Mont14} M. Alishahi, {\em On the dynamic coloring of graphs}, Discrete Appl. Math. 159 (2011) 152-156.

\bibitem{Mont15} A. Ahadi, S. Akbari, A. Dehghana, M. Ghanbari, {\em On the difference between chromatic number and dynamic chromatic number of graphs}, Discrete Math. 312 (17) (2012) 2579-2583.

\bibitem{BoMu01} J. A. Bondy, U. S. R. Murty, {\em Graph theory}, Springer, New York, 2008.

\bibitem{CFLSS15} Y. Chen, S. Fan, H.-J. Lai, H. Song, L. Sun, {\em On dynamic coloring for planar graphs and graphs of higher genus}, Discrete Appl. Math. 160 (2012) 1064-1071.

\bibitem{CFLX09} Y. Chen, S. Fan, H.-J. Lai, M. Xu, {\em Graph r-hued colorings-A survey}, Discrete Appl. Math. 321 (4) (2022) 24-48.

\bibitem{JLJZ07} X. H. Jia, F. X. Liu, B. Y. Ji, Z. H. Zhao, {\em The list r-hued coloring of $P_{5}$-free graph}, Applied Mathematics and Computation. 511 (2026) 129742.

\bibitem{MJVU21} M. Jovanovi$\acute{c}$, V. Uljarevi$\acute{c}$, {\em L-colorings of graphs}, Master thesis, University of Novi Sad, 2024.

\bibitem{LLMS05} H.-J. Lai, J. Lin, B. Montgomery, T. Shui, S. Fan, {\em Conditional coloring of graphs}, Discrete Math. 306 (16) (2006) 1997-2004.

\bibitem{LLMS08} H.-J. Lai, X. Lv, M. Xu,
{\em On r-hued colorings of graphs without short induced paths}, Discrete Math. 342 (7) (2019) 1904-1911.

\bibitem{LMP30} H.-J. Lai, B. Montgomery, H. Poon, {\em Upper bounds of dynamic chromatic number}, Ars Combin. 68 (1) (2003) 193-201.

\bibitem{SHI29} B. Montgomery, {\em Dynamic Coloring of Graphs}, Ph.D. thesis, West Virginia University, 2001.


\end{thebibliography}
\end{document}